\documentclass[12pt,reqno]{article}

\usepackage{amsmath,amsthm,amssymb,amsfonts}

\usepackage[dvipsnames]{xcolor}
\definecolor{webgreen}{rgb}{0,.5,0}
\definecolor{webbrown}{rgb}{.6,0,0}

\usepackage{graphicx}
\usepackage{tikz}
\usepackage{tikz-3dplot}

\usepackage{fullpage}
\usepackage{float}

\usepackage{mathtools}
\usepackage{enumitem}
\usepackage{authblk}
\usepackage{etoolbox} 

\usepackage[
  colorlinks=true,
  linkcolor=webgreen,
  filecolor=webbrown,
  citecolor=webgreen,
  urlcolor=webbrown
]{hyperref}

\theoremstyle{plain}
\newtheorem{theorem}{Theorem}
\newtheorem{corollary}[theorem]{Corollary}
\newtheorem{lemma}[theorem]{Lemma}
\newtheorem{proposition}[theorem]{Proposition}

\theoremstyle{definition}

\newtheorem{example}{Example}
\newcommand{\setEnvironmentQed}[2]{
  \AtBeginEnvironment{#1}{%
    \pushQED{\qed}\renewcommand{\qedsymbol}{#2}%
  }
  \AtEndEnvironment{#1}{\popQED}
}
\setEnvironmentQed{example}{\ensuremath{\blacksquare}}

\theoremstyle{remark}

\newtheorem*{remark*}{Remark}

\begin{document}

\title{Mathematics of Gozinta Boxes}

\author{Dillan Agrawal}
\author{Selena Ge}
\author{Jate Greene}
\author{Dohun Kim}
\author{Rajarshi Mandal}
\author{Tanish Parida}
\author{Anirudh Pulugurtha}
\author{Gordon Redwine}
\author{Soham Samanta}
\author{Albert Xu}
\affil{PRIMES STEP}
\author{Tanya Khovanova}
\affil{MIT}

\maketitle

\begin{abstract}
We study the geometric aspects of the magic trick called Gozinta Boxes. We generalize Gozinta Boxes to other dimensions, and we show that in three and higher dimensions, the maximum number of boxes is 3, and in two dimensions, the maximum is 4. We discuss other properties of Gozinta Boxes and provide a plethora of examples.
\end{abstract}

\section{Introduction}

\subsection{History and Background}

In 1966, a famous Czech magician named Lubor Fiedler thought of the concept of Gozinta Boxes while moving packages around in his apartment. He shared this idea with the rest of the world at the 1970 International Federation of Magic Societies convention in Amsterdam \cite{PP}. 

The original Gozinta Boxes trick involves a pair of boxes that ``magically'' fit inside each other. The trick begins with a red box placed inside of a black box, making the black box appear larger than the red box. Then, the black box is placed inside the red box, producing a paradoxical effect. The boxes are then shown to be exactly the same size, enhancing the effect. In reality, the outside box is not fully closed and is a rotated version of the inside box. Figure~\ref{fig:3dbox} shows a closed box on the left and an expanded box on the right.

\begin{figure}
\begin{center}
\tdplotsetmaincoords{80}{110}    
\begin{tikzpicture}[tdplot_main_coords, line join=round, scale=0.55]
\coordinate (O)   at (0,0,0);
\coordinate (X)   at (3.5,0,0);
\coordinate (Y)   at (0,4,0);
\coordinate (Z)   at (0,0,4.5);
\coordinate (XY)  at (3.5,4,0);
\coordinate (XZ)  at (3.5,0,4.5);
\coordinate (YZ)  at (0,4,4.5);
\coordinate (XYZ) at (3.5,4,4.5);
\draw[black, thick] (X)--(XY)--(Y);   
\draw[black, thick] (Z)--(YZ)--(Y);          
\draw[black, thick] (X)--(XZ)--(XYZ)--(XY);       
\draw[black, thick] (Z)--(XZ); \draw[black, thick] (YZ)--(XYZ); 

\coordinate (o)   at (-0.2,-0.2,0.5);
\coordinate (x)   at (3.7,-0.2,0.5);
\coordinate (y)   at (-0.2,4.2,0.5);
\coordinate (z)   at (-0.2,-0.2,4.7);
\coordinate (xy)  at (3.7,4.2,0.5);
\coordinate (xz)  at (3.7,-0.2,4.7);
\coordinate (yz)  at (-0.2,4.2,4.7);
\coordinate (xyz) at (3.7, 4.2, 4.7);
\draw[black, thick] (x)--(xy)--(y);   
\draw[black, thick] (z)--(yz)--(y);          
\draw[black, thick] (x)--(xz)--(xyz)--(xy);       
\draw[black, thick] (z)--(xz); \draw[black, thick] (yz)--(xyz); 
\definecolor{lightgray}{rgb}{0.8,0.8,0.8}
\definecolor{darkgray}{rgb}{0.3,0.3,0.3}
    \fill[lightgray,opacity=0.5] (3.5,0,0) -- (3.5,4,0) -- (3.5,4,4.5) -- (3.5,0,4.5) -- cycle;
    \fill[lightgray,opacity=0.5] (0,4,0) -- (3.5,4,0) -- (3.5,4,4.5) -- (0,4,4.5) -- cycle;
    \fill[lightgray,opacity=0.5] (0,4,4.5) -- (3.5,4,4.5) -- (3.5,0,4.5) -- (0,0,4.5) -- cycle;

    \fill[darkgray,opacity=0.5] (3.7,-0.2,0.5) -- (3.7,4.2,0.5) -- (3.7,4.2,4.7) -- (3.7,-0.2,4.7) -- cycle;
    \fill[darkgray,opacity=0.5] (-0.2,4.2,0.5) -- (3.7,4.2,0.5) -- (3.7,4.2,4.7) -- (-0.2,4.2,4.7) -- cycle;
    \fill[darkgray,opacity=0.5] (-0.2,4.2,4.7) -- (3.7,4.2,4.7) -- (3.7,-0.2,4.7) -- (-0.2,-0.2,4.7) -- cycle;
\coordinate (O1)   at (0,8,-1.5);
\coordinate (X1)   at (3.5,8,-1.5);
\coordinate (Y1)   at (0,12,-1.5);
\coordinate (Z1)   at (0,8,3);
\coordinate (XY1)  at (3.5,12,-1.5);
\coordinate (XZ1)  at (3.5,8,3);
\coordinate (YZ1)  at (0,12,3);
\coordinate (XYZ1) at (3.5,12,3);
\draw[black, thick] (X1)--(XY1)--(Y1);   
\draw[black, thick] (Z1)--(YZ1)--(Y1);          
\draw[black, thick] (X1)--(XZ1)--(XYZ1)--(XY1);       
\draw[black, thick] (Z1)--(XZ1); \draw[black, thick] (YZ1)--(XYZ1); 

\coordinate (o1)   at (-0.2,7.8,2.8);
\coordinate (x1)   at (3.7,7.8,2.8);
\coordinate (y1)   at (-0.2,12.2,2.8);
\coordinate (z1)   at (-0.2,7.8,7);
\coordinate (xy1)  at (3.7,12.2,2.8);
\coordinate (xz1)  at (3.7,7.8,7);
\coordinate (yz1)  at (-0.2,12.2,7);
\coordinate (xyz1) at (3.7, 12.2, 7);
\draw[black, thick] (x1)--(xy1)--(y1);   
\draw[black, thick] (z1)--(yz1)--(y1);          
\draw[black, thick] (x1)--(xz1)--(xyz1)--(xy1);       
\draw[black, thick] (z1)--(xz1); \draw[black, thick] (yz1)--(xyz1); 
\definecolor{lightgray}{rgb}{0.8,0.8,0.8}
\definecolor{darkgray}{rgb}{0.3,0.3,0.3}
    \fill[lightgray,opacity=0.5] (3.5,8,-1.5) -- (3.5,12,-1.5) -- (3.5,12,3) -- (3.5,8,3) -- cycle;
    \fill[lightgray,opacity=0.5] (0,12,-1.5) -- (3.5,12,-1.5) -- (3.5,12,3) -- (0,12,3) -- cycle;
    \fill[lightgray,opacity=0.5] (0,12,3) -- (3.5,12,3) -- (3.5,8,3) -- (0,8,3) -- cycle;

    \fill[darkgray,opacity=0.5] (3.7,7.8,2.8) -- (3.7,12.2,2.8) -- (3.7,12.2,7) -- (3.7,7.8,7) -- cycle;
    \fill[darkgray,opacity=0.5] (-0.2,12.2,2.8) -- (3.7,12.2,2.8) -- (3.7,12.2,7) -- (-0.2,12.2,7) -- cycle;
    \fill[darkgray,opacity=0.5] (-0.2,12.2,7) -- (3.7,12.2,7) -- (3.7,7.8,7) -- (-0.2,7.8,7) -- cycle;
\end{tikzpicture}
\end{center}
\caption{3D Gozinta Box: closed and expanded}
\label{fig:3dbox}
\end{figure}
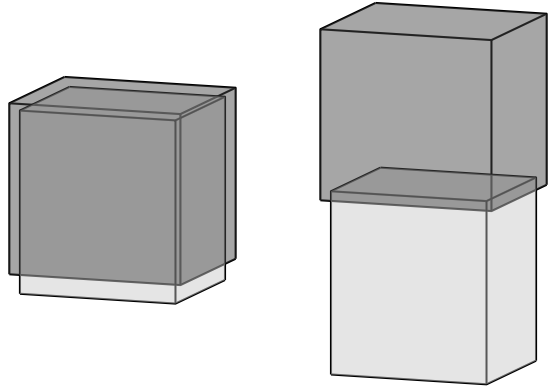

Figure~\ref{fig:cross-section} shows a 2D cross-section of two Gozinta Boxes when one box is inside the other. The outside box is shown using thicker lines.

\begin{figure}
\begin{center}
\begin{tikzpicture}[thick, scale=0.55]
  \draw[thick] (0,0) -- (0,3.6);
  \draw[thick] (0,3.6) -- (5.6,3.6);
  \draw[thick] (5.6,3.6) -- (5.6,0);

  \draw[thick] (0.3,0) -- (0.3,3);
  \draw[thick] (5.3,0) -- (0.3,0);
  \draw[thick] (5.3,3) -- (5.3,0);

  \def\yshift{-3.5}
  \draw[ultra thick] (3,3.8) -- (6,3.8);
  \draw[ultra thick] (6,3.8) -- (6,-0.2);
  \draw[ultra thick] (6,-0.2) -- (3,-0.2);

  \draw[ultra thick] (-0.2,-0.5) -- (-0.2,4.1);
  \draw[ultra thick] (3.4,-0.5) -- (-0.2,-0.5);
  \draw[ultra thick] (3.4,4.1) -- (-0.2,4.1);
\end{tikzpicture}
\end{center}
\caption{A cross-section of two Gozinta Boxes}
\label{fig:cross-section}
\end{figure}
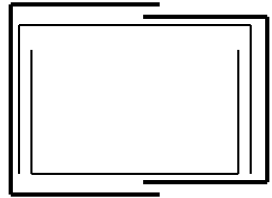

A pair of Gozinta Boxes was originally called Lubor Die, named after its creator Lubor Fiedler, and they were eventually given the name ``Gozinta'', which is an informal word derived from ``Goes-In-Da". They are sold under many other names, including Driebeck dice~\cite{SME} and In \& Out boxes~\cite{PM}. 

A three-box version of this trick was recently invented by Ivo David Oliveira. It is sold by TCC Magic under the name ``Triple Gozinta Boxes'' \cite{TCC}. There are three boxes: black, red, and white. Initially, the white box is inside the red box, which is inside the black box, implying that the white box is the smallest while the black box is the largest. Then, the boxes are put in reverse order: the black box inside the red box, which is inside the white box.

\subsection{Road Map}

We start with definitions and observations in Section~\ref{sec:io}. We initially assume that boxes are in two orders: the natural order and the reverse order. In Section~\ref{sec:whichsideexpands}, we discuss different options for which side may expand. If there are several boxes, we can assume that each box is closed in one of the orders and expanded in the other. We discuss different cases of boxes being closed and expanded in Section~\ref{sec:whichboxesareclosed}. We prove that it is impossible to have five or more Gozinta Boxes in any dimension.

In Section~\ref{sec:2D}, we concentrate on Gozinta Boxes in two dimensions. We show that in two dimensions, it is possible to have four Gozinta Boxes. In Section~\ref{sec:3D}, we move to three dimensions and show that four Gozinta Boxes are impossible in 3D. In Section~\ref{sec:ND}, we move to higher dimensions and show that four Gozinta Boxes are impossible.

In Section~\ref{sec:permutations}, we expand our discussion to considering not only the reverse order of the boxes but also other permutations we can achieve from the initial order. We show that starting from the natural order, all permutations for two or three boxes are possible. If there are 4 boxes, we show that all permutations except the reverse one are possible. We also discuss whether it might be possible to use the same set of boxes to achieve several permutations.

In Section~\ref{sec:FromTo}, we discuss how, given a set of Gozinta Boxes, to build another set with different dimensions that can perform the same trick. We apply the results in Section~\ref{sec:twoboxes}, where we divide pairs of Gozinta Boxes in 3D into four types.

\section{Initial observations}
\label{sec:io}

\subsection{Definitions}

We can view a closed Gozinta Box as a rectangular cuboid, a convex polyhedron with six rectangular faces. When we open a Gozinta Box, we get two parts; each is a rectangular cuboid with one face missing. When made in real life, the outer part is slightly larger than the inner one, but in this paper, we will ignore the difference. We also assume that the boxes have no thickness.

Note that if we can find a set of theoretical boxes with zero thickness, then it is possible to build real boxes with nonzero thickness: we can multiply the dimensions of all the boxes by a sufficiently large number to make the extra thickness of the boxes negligible.

We denote boxes with capital letters, for example, $A$, $B$, $C$, and $D$. We denote the side lengths of the closed boxes by the corresponding lowercase letters. For example, the side lengths of box $A$ are denoted $\vec{a} = (a_1,a_2,\ldots,a_n)$, where we assume that the side lengths are arranged in a non-decreasing order:
\[a_1 \leq a_2 \leq \dots \leq a_n.\] 
We also call the vector $\vec{a}$ the \textit{dimensions} of box $A$.

When side $a_i$ is expanded, we denote its length as $a'_i$. Recall that only one side of a Gozinta Box can expand. Also, when the side is expanded, the box should still look like a box. Thus, there is a limit to the expansion; namely, a side cannot expand to more than twice its original length, as this would cause both halves of the box to be separated. This gives us an \textit{expansion bound}:
\[a'_i \leq 2a_i.\]

Consider two $n$-dimensional vectors $\vec{a} = (a_1,a_2,\ldots,a_n)$ and $\vec{b} = (b_1,b_2,\ldots,b_n)$. We say that $\vec{a}$ \textit{dominates} $\vec{b}$ if $a_i \geq b_i$ for any $i$ in the range 1 to $n$, and we denote it as $\vec{a} \succeq \vec{b}$. The dimensions vector of the expanded box $A$ dominates $\vec{a}$. We say that $\vec{a}$ \textit{strictly dominates} $\vec{b}$ if $a_i > b_i$ for any $i$ in the range 1 to $n$, and we denote it as $\vec{a} \succ \vec{b}$. Vector $\vec{a}$ strictly dominates $\vec{b}$, if and only if box $B$ can fit inside closed box $A$.

\begin{example}
\label{ex:3-4-5}
In the original Gozinta Boxes, the two boxes are the same, with the dimensions of each box being roughly $(3,4,5)$. When the smallest side expands, the maximum size of each box is $(4,5,6)$, which can comfortably enclose the other closed box.
\end{example}

We sometimes use diagrams to show how boxes can fit into each other. In diagrams, we do not use the vector notation; we use the $\times$ symbol to separate dimensions. We also put boxes that are outside on top of boxes that are inside. To emphasize the expanded side, we put the original side length in parentheses. Example~\ref{ex:3-4-5} in diagram form will look like this:
\begin{center}
\begin{tabular}{r @{\,$\times$\,} r @{\,$\times$\,} l}
4 & 5 & 6(3) \\
3 & 4 & 5. \\
\end{tabular}
\end{center}

Recently, Ivo David Oliveira invented and started selling triple Gozinta Boxes. The dimensions in the example below are theoretical and are not intended to match the dimensions of the sold triple boxes.

\begin{example}
\label{ex:6-8-10}
We can have a box of dimensions $(3.5,4.5,5)$ that will fit between the two Gozinta Boxes in the example above. If we scale these dimensions, we can get all the boxes of integer sizes: two boxes of sizes $(6,8,10)$ and one box of size $(7,9,11)$:
\begin{center}
\begin{tabular}{r @{\,$\times$\,} c @{\,$\times$\,} l}
8 & 10 & 12(6) \\
7 & 9 & 11 \\
6 & 8 & 10.
\end{tabular}
\end{center}
\vspace{-2.05em}\qedhere
\end{example}

\subsection{Which side expands}
\label{sec:whichsideexpands}

It is natural to assume that the smallest side of the box should be the side that expands. However, not all boxes need to expand on the smallest side. Below is an example in which the middle side expands. This example also shows that the expanded side does not need to become the largest side.

\begin{example}
\label{ex:6-9-10}
Suppose we have a box $A$ with dimensions $(6, 9, 10)$ inside a box $B$ with dimensions $(7, 8, 11)$. In the first box, the smallest side needs to expand to make, for example, box $(9, 10, 12)$. In the second box, it could be the smallest or the middle side, making box $B$ expand to $(7, 11, 12)$ or $(8, 11, 12)$, respectively.
\end{example}

The following example has the largest side expanding. Moreover, considering limitations on the expansion, no other side can work.

\begin{example}
\label{ex:largestsideexpands}
Suppose box $A$ has dimensions $(5, 7, 999)$, expanding on the smallest dimension, and box $B$ has dimensions $(6, 8, 500)$. Then, in this case, box $B$ has to expand on the largest side:
\begin{center}
\noindent
\makebox[\textwidth][c]{
\begin{minipage}[t]{0.31\textwidth}
\centering
\begin{tabular}{c @{\hspace{0.5em}} r @{\,$\times$\,} c @{\,$\times$\,} l}
$B:$ & 6 & 8 & 1000(500) \\
$A:$ & 5 & 7 & 999          
\end{tabular}
\end{minipage}
\hspace{0.06\textwidth}  
\begin{minipage}[t]{0.31\textwidth}
\centering
\begin{tabular}{c @{\hspace{0.5em}} r @{\,$\times$\,} c @{\,$\times$\,} l}
$A:$ & 7 & 9(5) & 999  \\
$B:$ & 6 & 8 & 500.  
\end{tabular}
\end{minipage}
}
\end{center}
\vspace{-2.05em}\qedhere
\end{example}

Another common assumption is that when a side expands, that side has to become the largest side of the box. However, in the previous example, side $a_1$ expands to become the second-largest side. Also, notice that in this example, any side of $B$ can expand.

\begin{example}
\label{ex:5-7-11versus6-8-10}
Suppose box $A$ has dimensions $(5, 7, 11)$ and box $B$ has dimensions $(6, 8, 10)$. In the natural order, any side of box $B$ can expand to 12:
\begin{center}
\noindent
\makebox[\textwidth][c]{
\begin{minipage}[t]{0.31\textwidth}
\centering
\begin{tabular}{c @{\hspace{0.5em}} r @{\,$\times$\,} c @{\,$\times$\,} l}
$B:$ & 8 & 10 & 12(6) \\
$A:$ & 5 & 7 & 11.
\end{tabular}
\end{minipage}
\hspace{0.06\textwidth}
\begin{minipage}[t]{0.31\textwidth}
\centering
\begin{tabular}{c @{\hspace{0.5em}} r @{\,$\times$\,} c @{\,$\times$\,} l}
$B:$ & 6 & 10 & 12(8) \\
$A:$ & 5 & 7 & 11
\end{tabular}
\end{minipage}
\hspace{0.06\textwidth}  
\begin{minipage}[t]{0.31\textwidth}
\centering
\begin{tabular}{c @{\hspace{0.5em}} r @{\,$\times$\,} c @{\,$\times$\,} l}
$B:$ & 6 & 8 & 12(10) \\
$A:$ & 5 & 7 & 11.
\end{tabular}
\end{minipage}
}
\end{center}
In the reverse order, side $a_1$ can expand to 9:
\begin{center}
\begin{tabular}{c @{\hspace{0.5em}} r @{\,$\times$\,} c @{\,$\times$\,} l}
$A:$ & 7 & 9(5) & 11 \\
$B:$ & 6 & 8 & 10.
\end{tabular}
\end{center}
\vspace{-2.05em}\qedhere
\end{example}

We can also provide an example where $a_1$ can expand and stay the smallest. The following example has $2a_1 < a_2$, which means that if side $a_1$ expands, it has to stay the smallest.

\begin{example}
\label{ex:5-11-13versus7-10-12}
Suppose box $A$ has dimensions $(5, 11, 13)$ and box $B$ has dimensions $(7, 10, 12)$. In the natural order, box $B$ expands side $b_1$ to 14, and in the reverse order, $a_1$ expands to 8.
\end{example}

These examples challenged our intuition. Still, our initial intuition is not completely off.

\begin{proposition}
If there are $n$ boxes that can fit into each other in the natural and the reverse order, then at least $n-1$ of them have to expand on the smallest side.
\end{proposition}

\begin{proof}
Consider any two boxes $A$ and $B$. In one of the orders, $A$ is inside $B$, and in the other order, $B$ is inside $A$. There may be other boxes in between. Suppose $\ell$ is the length of the smallest side among all six sides of these two boxes. Without loss of generality, assume that box $A$ has side $\ell$, in other words, $a_1 = \ell$. Then box $B$, either closed or expanded, needs to fit inside $A$, which means that side $a_1$ needs to expand.

For every pair of boxes, the smallest side among all of the sides must expand. Thus, all boxes, except the one with the largest smallest side length, must expand on the smallest side.
\end{proof}

\subsection{Which boxes are closed}
\label{sec:whichboxesareclosed}

Until Section~\ref{sec:permutations}, we assume that boxes are in two orders: the natural order and the reverse order. For example, when we have three boxes, the natural order is when $A$ is inside $B$ and $B$ is inside $C$. The reverse order is when $C$ is inside $B$ and $B$ is inside $A$. In Section~\ref{sec:permutations}, we look at other permutations of the boxes.

Each box has two sets of dimensions: the first set corresponds to the side lengths when the boxes are in natural order, and the second set corresponds to the side lengths when the boxes
are in reverse order (or another permutation). If a box has two different sets of dimensions, consider the dimensions that produce a smaller volume. Without loss of generality, we can assume that this set of dimensions corresponds to the closed box. The other set of dimensions corresponds to the expanded box. If a box has the same dimensions in both orders, it is natural to assume that this box is closed in both orders. In Example~\ref{ex:6-8-10}, the box with dimensions $(7,9,11)$ does not need to be expanded in either order. However, to reduce the number of cases we study in this section, we can always assume that in one of these orders, we can expand the box by a tiny amount. This leads to the following observation.

\textbf{Observation.} We can assume that each box is closed in one order and expanded in the other.

We can also assume that the innermost box is closed. Our observation implies that the outermost box is always expanded.

There is a connection between which boxes are closed and which sides expand.

\begin{lemma}
\label{lem:closedexpands}
If a box $A$ is inside a closed box $B$ in one order, then the smallest side of $A$ has to expand to become the largest when the order of the boxes is reversed.
\end{lemma}

\begin{proof}
If box $A$ is inside a closed box $B$ in some order, then $\vec{b} \succ \vec{a}$, implying that $a_i < b_i$ for any $i$. It follows that $a_1 < b_i$ for any $i$, which means $a_1$ needs to expand, but we also need to have a side bigger than $b_3$ in the expanded version of $A$, and $b_3 > a_3$. Thus, the expanded box $A$ has sorted sides $a_2 \leq a_3 \leq a'_1$. Moreover, we get the following inequality chain as a result:
\[a_1 < b_1 < a_2 < b_2 < a_3 < b_3 < a'_1.\qedhere\]
\end{proof}

Let us denote a closed box by $C$ and an expanded box by $E$.

\begin{theorem}
\label{thm:no5}
It is impossible to have a set of 5 Gozinta Boxes in any dimension $n \ge 2$. Also, if there is a set of four boxes, the boxes must follow the pattern CCEE in both the natural and the reverse orders. 
\end{theorem}

\begin{proof}
Suppose we have a set of more than 4 Gozinta Boxes. By the pigeonhole principle, there must be at least $3$ boxes that are all closed in either the natural order or the reversed order. Without loss of generality, assume that three boxes $A$, $B$, and $C$ are all closed in the natural order, where $A$ is inside $B$ inside $C$. It follows that $\vec{a} \prec \vec{b} \prec \vec{c}$. In particular, we have
\[a_1 < b_1 <c _1 \quad \textrm{ and } \quad a_2 < b_2 < c_2.\]

In the reverse order, box $B$ has to fit inside of $A$. By Lemma~\ref{lem:closedexpands}, side $a_1$ expands to be the largest dimension of $A$. Similarly, box $C$ has to fit inside $B$. The only way to do this is if $b_1$ expands to be the largest dimension of $B$. Thus, in reverse order, the smallest dimension of $A$ is $a_2$, and the smallest dimension of $B$ is $b_2$. As $B$ must fit inside $A$, we have $b_2 < a_2$, contradicting the previous inequality.

Consider 4 boxes. As we mentioned before, we can assume that the innermost box is always closed and the outermost box is always expanded. In addition, from the argument above, we cannot have 3 closed boxes in any order. Thus, we have two options left for the natural order:
\begin{enumerate}
	\item $CECE$: Boxes $A$ and $C$ are closed.
	\item $CCEE$: Boxes $A$ and $B$ are closed, and boxes $C$ and $D$ are expanded.
\end{enumerate}

Consider the first case. Since in the natural order, closed $C$ contains expanded $B$, then closed $B$, being smaller than expanded $B$, must also fit inside closed $C$. A similar argument works for the reverse order, implying that closed $C$ fits inside closed $B$. This creates a contradiction, as the first part implies that the volume of closed $B$ is smaller than the volume of closed $C$, while the second part implies the opposite. This case is impossible, leaving $CCEE$ as the only remaining case.
\end{proof}

\section{2D Gozinta Boxes}
\label{sec:2D}

Magicians have a 2D trick called \textit{Gozinta envelopes} \cite{PP}, available for purchase at Funtime Magic \cite{FM}. There are two envelopes, red and purple, with a black border. The envelopes are of the same size. The magician can seemingly put the red envelope inside the purple one and the purple envelope inside the red one. However, this is not the case. When the red envelope is inside the purple one, one side sticks out. The magician obscures this fact, and the black border helps with this illusion. Gozinta envelopes are not a true variation of Gozinta Boxes.

In this section, we discuss 2D Gozinta Boxes, which we call \textit{Gozinta Rectangles}. Figure~\ref{fig:2d} shows a closed and expanded rectangle.

\begin{figure}[ht!]
\begin{center}
\begin{tikzpicture}[thick, scale=0.55]

\definecolor{lightgray}{rgb}{0.8,0.8,0.8}
\definecolor{darkgray}{rgb}{0.3,0.3,0.3}

\begin{scope}[xshift=0cm]
  \fill[lightgray, opacity=0.5] (0.3,-0.3) -- (0.3,3.0) -- (5.3,3.0) -- (5.3,-0.3) -- cycle;
  \fill[darkgray, opacity=0.5] (0,0) -- (0,3.6) -- (5.6,3.6) -- (5.6,0) -- cycle;

  \draw[black, thick] (0.3,-0.3) -- (0.3,3.0);
  \draw[black, thick] (5.3,-0.3) -- (0.3,-0.3);
  \draw[black, thick] (5.3,3.0) -- (5.3,-0.3);

  \draw[black, thick] (0,0) -- (0,3.6);
  \draw[black, thick] (0,3.6) -- (5.6,3.6);
  \draw[black, thick] (5.6,3.6) -- (5.6,0);
\end{scope}

\begin{scope}[xshift=8cm, yshift=1.2cm]
  \fill[lightgray, opacity=0.5] (0.3,-2.9) -- (0.3,0.4) -- (5.3,0.4) -- (5.3,-2.9) -- cycle;
  \fill[darkgray, opacity=0.5] (0,0) -- (0,3.6) -- (5.6,3.6) -- (5.6,0) -- cycle;

  \draw[black, thick] (0.3,-2.9) -- (0.3,0.4);
  \draw[black, thick] (5.3,-2.9) -- (0.3,-2.9);
  \draw[black, thick] (5.3,0.4) -- (5.3,-2.9);

  \draw[black, thick] (0,0) -- (0,3.6);
  \draw[black, thick] (0,3.6) -- (5.6,3.6);
  \draw[black, thick] (5.6,3.6) -- (5.6,0);
\end{scope}

\end{tikzpicture}
\end{center}
\caption{2D Gozinta Rectangle: closed and expanded}
\label{fig:2d}
\end{figure}
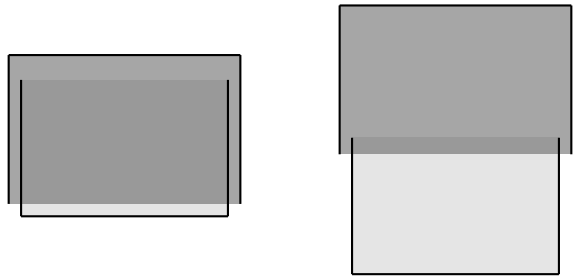

Here, we present an example of Gozinta Rectangles which are a 2D version of 3D Example~\ref{ex:6-8-10} with the largest side ignored.

\begin{example}
We have two rectangles $A$ and $C$ that are $(6,8)$ when closed, and the smallest side expands to get dimensions $(8,10)$ when expanded. Rectangle $B$ is $(7,9)$. When we consider $A$ and $C$ separately, we get two rectangles that are the same; both of them expand on the smallest side. When we consider rectangles $A$ and $B$, we get a case where one of the rectangles is closed in both configurations.
\end{example}

Here is an example of 4 Gozinta Rectangles.

\begin{example}
Let the four rectangles $A$, $B$, $C$, and $D$ have dimensions $(8,14)$, $(9,15)$, $(11,13)$, and $(10,12)$, correspondingly. All rectangles except for $C$ expand on the smaller side, and $C$ expands on the larger side. Our two configurations are below, with the left configuration in the natural order and the right one in the reverse order:
\begin{center}
\begin{minipage}{0.31\textwidth}
\centering
\begin{tabular}{c @{\hspace{0.5em}} r @{\,$\times$\,} l }
$D:$ & 12 & 20(10) \\
$C:$ & 11 & 16(13) \\
$B:$ &  9 & 15 \\
$A:$ &  8 & 14
\end{tabular}
\end{minipage}
\hspace{0.06\textwidth}  
\begin{minipage}{0.31\textwidth}
\centering
\begin{tabular}{c @{\hspace{0.5em}} r @{\,$\times$\,} l }
$A:$ & 14 & 16(8) \\
$B:$ & 12(9) & 15 \\
$C:$ & 11 & 13 \\
$D:$ & 10 & 12.
\end{tabular}
\end{minipage}
\end{center}
\vspace{-2.05em}\qedhere
\end{example}

Examples with identical boxes/rectangles are more elegant, and thus are desired. In both sets (a pair and a triple) of Gozinta Boxes used in magic, the outermost and the innermost boxes in the natural order are the same. Here is an example of such a case for four Gozinta Rectangles.

\begin{example}
It is possible for rectangles $A$ and $D$ in a set of four Gozinta Rectangles to have the same dimensions. We show the diagrams below, where the natural order is on the left, and the reverse order is on the right. Here, both rectangles $A$ and $D$ have dimensions $(9,13)$:
\begin{center}
\begin{minipage}{0.31\textwidth}
\centering
\begin{tabular}{c @{\hspace{0.5em}} r @{\,$\times$\,} l }
$D:$ & 13 & 16(9) \\
$C:$ & 12(10) & 15 \\
$B:$ &  11 & 14 \\
$A:$ &  9 & 13
\end{tabular}
\end{minipage}
\hspace{0.06\textwidth}  
\begin{minipage}{0.31\textwidth}
\centering
\begin{tabular}{c @{\hspace{0.5em}} r @{\,$\times$\,} l }
$A:$ & 13 & 17(9)\\
$B:$ & 11 & 16(14)\\
$C:$ & 10 & 15 \\
$D:$ & 9 & 13.
\end{tabular}
\end{minipage}
\end{center}
\vspace{-2.05em}\qedhere
\end{example}

To get to full symmetry, one might want to have $A=D$ and $B=C$. However, this is impossible. Moreover, rectangle $B$ cannot be the same as rectangle $C$ in a set of four Gozinta Rectangles.

\begin{proposition}
If we have a set of 4 Gozinta Rectangles, then the two middle rectangles cannot be the same.
\end{proposition}

\begin{proof}
By Theorem~\ref{thm:no5}, we only need to consider one case, depending on which rectangles are expanded in the natural order: CCEE.

If rectangles $B$ and $C$ are the same, we need the smallest side of each of them to expand and become the largest in reverse order. We also know that $a_1$ needs to expand as it is smaller than both $b_1$ and $b_2$. In addition, side $a_1$ needs to become the largest side of all the rectangles, so rectangle $A$, when expanded, is $(a_2,a'_1)$. Thus, we can draw the following diagrams, where the natural order is on the left, and the reverse is on the right, and rectangle $D$ is excluded:
\begin{center}
\begin{minipage}{0.31\textwidth}
\centering
\begin{tabular}{c @{\hspace{0.5em}} r @{\,$\times$\,} l }
$C:$ & $b_2$ & $b'_1$ \\
$B:$ & $b_1$ & \boldmath$b_2$ \\
$A:$ & $a_1$ & \boldmath$a_2$
\end{tabular}
\end{minipage}
\hspace{0.06\textwidth}  
\begin{minipage}{0.31\textwidth}
\centering
\begin{tabular}{c @{\hspace{0.5em}} r @{\,$\times$\,} l }
$A:$ & \boldmath$a_2$ & $a'_1$ \\
$B:$ & \boldmath$b_2$ & $b'_1$ \\
$C:$ & $b_1$ & $b_2$.
\end{tabular}
\end{minipage}
\end{center}

Now we get a contradiction, which is bolded in the diagram. From the natural order, we have $a_2 < b_2$, and from the reverse order, we get $b_2 > a_2$.
\end{proof}

\section{3D Gozinta Boxes}
\label{sec:3D}

In the previous section, we saw that it is possible to have 4 Gozinta Boxes in 2D. What about 3D?

\begin{theorem}
Four Gozinta Boxes in 3D are impossible.
\end{theorem}

\begin{proof}
From Theorem~\ref{thm:no5}, we know that if 4 boxes exist, then in the natural and reverse orders, the pattern must be CCEE. In other words, in the natural order, the two innermost boxes are closed $A$ inside closed $B$, while in the reverse order, the two innermost boxes are closed $D$ inside closed $C$. By Lemma~\ref{lem:closedexpands}, the smallest sides of $A$ and $D$ must expand to become the largest. Thus, $A$ expands to $a_2 \leq a_3 \leq a'_1$ and $D$ expands to $d_2 \leq d_3 \leq d'_1$.

Now we consider box $B$. When it is expanded, its smallest side has to be smaller than $a_2$. However, from the natural order of the boxes, we know that $b_2 > a_2$. This means the smallest side of the expanded box $B$ has to be either $b_1$ or expanded $b_1$, denoted as $b_1'$. Similarly, the second smallest side is $b_2$ or $b_2'$, and the largest side is $b_3$ or $b_3'$. In other words, although we do not know which side expands, the relative order among the sides is preserved after the expansion. A similar argument applies to box $C$.

We represent the natural order as the diagram on the left and the reverse order as the one on the right:
\begin{center}
\makebox[\textwidth][c]{
\begin{minipage}[t]{0.31\textwidth}
\centering
\begin{tabular}{c @{\hspace{0.5em}} r @{\,$\times$\,} c @{\,$\times$\,} l}
$D:$ & $d_2$      & $d_3$      & $d_1'$ \\
$C:$ & $c_1/c_1'$ & $c_2/c_2'$ & $c_3/c_3'$ \\
$B:$ & $b_1$      & $b_2$      & $b_3$ \\
$A:$ & $a_1$      & $a_2$      & $a_3$ 
\end{tabular}
\end{minipage}
\hspace{0.06\textwidth}  
\begin{minipage}[t]{0.31\textwidth}
\centering
\begin{tabular}{c @{\hspace{0.5em}} r @{\,$\times$\,} c @{\,$\times$\,} l}
$A:$ & $a_2$      & $a_3$      & $a_1'$ \\
$B:$ & $b_1/b_1'$ & $b_2/b_2'$ & $b_3/b_3'$ \\
$C:$ & $c_1$      & $c_2$      & $c_3$ \\
$D:$ & $d_1$      & $d_2$      & $d_3$.
\end{tabular}
\end{minipage}
} 
\end{center}

With three coordinates and at most one expanding side per box, the pigeonhole principle implies that for some $1\le j \le 3$, neither $b_j$ nor $c_j$ expands, so we have both $b_j < c_j$ and $c_j < b_j$, reaching a contradiction. Thus, $4$ or more boxes are impossible in 3D.
\end{proof}

\section{nD Gozinta Boxes}
\label{sec:ND}

We know that the maximum number of boxes in 2D is 4, and in 3D is 3. What about higher dimensions?

\begin{proposition}
\label{prop:removeddimension}
If there exist $k$ Gozinta Boxes in $n$ dimensions, then there exist $k$ Gozinta Boxes in $n-1$ dimensions.
\end{proposition}

\begin{proof}
Suppose we have $k$ boxes in $n$ dimensions. Consider one of the boxes, say $A$, with sides $a_1 \leq a_2 \leq \dots \leq a_n$. We build a new box $\mathcal{A}$ in $n-1$ dimensions with sides $a_1 \leq a_2 \leq \dots \leq a_{n-1}$.

Suppose box $A$ has sides $\alpha_1 \leq \alpha_2 \leq \dots \leq \alpha_n$ when expanded. Here, $\alpha_i$ could equal $a_i$, $a_{i+1}$, or $a_j'$. Suppose $\alpha_n = a_n$ or $\alpha_n = a_n'$. Then, our new box $\mathcal{A}$ has sides $\alpha_1 \leq \alpha_2 \leq \dots \leq \alpha_{n-1}$ when expanded. In other words, we ignore the largest side. Suppose $\alpha_n = a_j'$, then $A$ has side $j$ that expands. To build box $\mathcal{A}$, we let $a_j$ expand to $a_n$ instead of $a_j'$. Note that $a_j \leq a_n \leq a_j'$, implying that such an expansion is possible. If $\alpha_n = a_j'$, then $\alpha_{n-1} = a_n$. So, for box $\mathcal{A}$, we can replace $\alpha_{n-1} = a_n$ with $\alpha_{n-1} = a_j'$. Hence, it is possible to expand $\mathcal{A}$ to $\alpha_1 \leq \alpha_2 \leq \dots \leq \alpha_{n-1}$.

We replaced boxes in $n$ dimensions with new boxes such that all the sides except the largest one have the same set of values as in the original boxes in both the natural and reversed order. This means the resulting boxes satisfy all the necessary inequalities.
\end{proof}

\begin{corollary}
In $n > 3$ dimensions, we cannot have more than 3 boxes.
\end{corollary}

In $n>3$ dimensions, we can build $3$ Gozinta Boxes in a similar manner as in Example~\ref{ex:6-8-10} for $3$ dimensions: The outermost and innermost boxes have dimensions $(n, n+1, \dots, 2n-2, 2n-1)$ and expand on their smallest side. The box in between them has dimensions $(n+0.5, n+1.5, \dots, 2n-1.5, 2n-0.5)$ and does not need to expand. In both the natural and reverse order, one box expands to have integer dimensions $n+1$ to $2n$, the middle box has dimensions from $n+0.5$ to $2n-0.5$, and the smallest box has dimensions from $n$ to $2n-1$. Here, the 3 boxes are shown as a diagram, where we added brackets around dimensions for readability: 
\begin{center}
\begin{tabular}{r @{\,$\times$\,} c @{\,$\times$\,} c @{\,$\times$\,} c @{\,$\times$\,} c @{\,$\times$\,} l}
$[n+1]$ & $[n+2]$ & $[n+3]$ & $\cdots$ &  $[2n-1]$ & $[2n(n)]$\\
$[n+0.5]$ & $[n+1.5]$ & $[n+2.5]$ & $\cdots$ & $[2n-1.5]$ & $[2n-0.5]$ \\
$[n]$ & $[n+1]$ & $[n+2]$ & $\cdots$ &  $[2n-2]$ & $[2n-1]$.\\      
\end{tabular}
\end{center}

\section{Permutations}
\label{sec:permutations}

Suppose our initial order is the natural order; that is, the order of the boxes corresponds to the identity permutation. For which permutations $p$ on $n$ elements can we prove that there exist boxes that can fit in the natural order and in order $p$? We call a permutation $p$ for which it is possible an \textit{achievable} permutation.

We define the \textit{coolness} of the trick that transforms the natural order into permutation $p$ as the number of inversions of $p$. Tricks with a higher number of inversions look more impressive.

We also ask a different question: given a set of boxes in the natural order, what other permutations can be achieved with the same set of boxes? We call such sets of permutations \textit{jointly-achievable}.

\subsection{Achievable permutations}

We can construct achievable permutations with a larger number of boxes by using achievable permutations with a smaller number of boxes. Here, we use the standard notation $S_n$ to denote the permutation group of $n$ elements.

\begin{theorem}
\label{thm:boost}
	If permutations $p_1 \in S_n$ and $p_2 \in S_k$ are achievable, then we can build an achievable permutation $p' \in S_{n+k}$ by using one of the following operations.
	\begin{enumerate}[label=Operation \arabic*:, leftmargin=*, labelindent=1em]
	\item Add $n$ to all elements of permutation $p_2$ and append the result at the end of $p_1$.
	\item Replace element $x$ in $p$ with a pair $x(x+1)$ and increase all other elements greater than $x$ by 1.
	\item If permutation $p$ is achievable, then so is its inverse $p^{-1}$.
	\end{enumerate}
\end{theorem}

\begin{proof}
Operation 1 is equivalent to the following. Consider two sets of Gozinta Boxes. The first set can be permuted according to $p_1$ and the second according to $p_2$. Then, we can add a constant to each side in the second set so that the smallest side in the second set is larger than the largest closed or expanded side of the first set. We can then arrange the first set inside the second set in any permutation that is independently allowed for each set.

Consider Operation 2. Suppose $d$ is the smallest nonzero difference between any two sides of every pair of boxes. Choose $\epsilon < d$. Consider box number $x$ in natural order with sides $x_i$, where $i$ ranges from 1 to $n$. We can build a new box with sides $x_i+\epsilon$. The new box can contain box number $x$ as each side is slightly longer. It can also fit inside any box that was able to fit box number $x$ inside.

For Operation 3, assume that the permutation $p$ is achievable. Now, instead of assuming that the identity corresponds to the natural order, we can assume that $p$ corresponds to the natural order. We relabel the boxes so that permutation $p$ becomes the identity. We apply the inverse $p^{-1}$ to get back to the original identity permutation. Thus, the same set of boxes that make $p$ achievable, after relabeling, show that $p^{-1}$ is achievable.
\end{proof}

Now, we show which permutations of orders 2, 3, and 4 are achievable.

\begin{proposition}
All permutations of orders 2 and 3 are achievable. For permutations of order 4, only the reverse permutation is not achievable.
\end{proposition}

\begin{proof}
We know that permutations 21 and 321 are achievable. They correspond to a pair and a triple of real Gozinta Boxes, respectively. Now we use Theorem~\ref{thm:boost} to find other achievable permutations.

For 3 Gozinta Boxes, we want to check permutations 132, 213, 231, and 312. To do so, we use the fact that permutation 21 is achievable. Then permutation 132 is achievable by using Operation 1 on permutations 1 and 21 and permutation 213 is achievable by using Operation 1 on permutations 21 and 1. Permutations 231 and 312 are achievable by Operation 2. To get permutation 231, we replace 2 in 21 with 32. To get permutation 312, we replace 1 in 21 with 12 and increase 2 at the beginning to 3.

Now we look at four boxes. By Theorem~\ref{thm:boost} and the fact that any permutation of order 3 is achievable, we can achieve permutations that start with 1 or end with 4 by Operation 1. The former is equivalent to $p_1 = 1$, and the latter to $p_2 = 1$. If $p_1 = 21$ and $p_2 = 21$, we can achieve permutation 2143. 
    
By Operation 2, we can achieve permutations that have two consecutive numbers next to each other in order: 2341, 3412, 3421, 4123, 4231, and 4312. We are left with 7 permutations: 2413, 2431, 3142, 3241, 4132, 4213, and 4321. We now check which of them are reverses of each other. We have $2413^{-1} = 3124$, $2431^{-1} = 4132$, $3241^{-1} = 4213$, and $4321^{-1} = 4321$.

We are left with 4 permutations to consider. We now show that three of them are achievable through examples. For permutation 2413, we have the following example:
\begin{center}
\noindent
\makebox[\textwidth][c]{
\begin{minipage}[t]{0.31\textwidth}
\centering
\begin{tabular}{c @{\hspace{0.5em}} r @{\,$\times$\,} c @{\,$\times$\,} l}
$D:$ & 15 & 19 & 23 \\
$C:$ & 14 & 18 & 22 \\
$B:$ & 13 & 17 & 21 \\
$A:$ & 12 & 16 & 20          
\end{tabular}
\end{minipage}
\hspace{0.06\textwidth}  
\begin{minipage}[t]{0.31\textwidth}
\centering
\begin{tabular}{c @{\hspace{0.5em}} r @{\,$\times$\,} c @{\,$\times$\,} l}
$C:$ & 18 & 22 & 25(14) \\
$A:$ & 16 & 20 & 24(12) \\
$D:$ & 15 & 19 & 23 \\
$B:$ & 13 & 17 & 21.        
\end{tabular}
\end{minipage}
}
\end{center}

For permutation 2431, we have the following example:
\begin{center}
\noindent
\makebox[\textwidth][c]{
\begin{minipage}[t]{0.31\textwidth}
\centering
\begin{tabular}{c @{\hspace{0.5em}} r @{\,$\times$\,} c @{\,$\times$\,} l}
$D:$ & 18  & 22 & 24(14) \\
$C:$ & 15    & 19   & 23 \\
$B:$ & 13    & 17   & 21 \\
$A:$ & 12    & 16   & 20          
\end{tabular}
\end{minipage}
\hspace{0.06\textwidth}  
\begin{minipage}[t]{0.31\textwidth}
\centering
\begin{tabular}{c @{\hspace{0.5em}} r @{\,$\times$\,} c @{\,$\times$\,} l}
$A:$ & 16    & 20   & 24(12) \\
$C:$ & 15    & 19   & 23 \\
$D:$ & 14  & 18 & 22 \\
$B:$ & 13    & 17   & 21.         
\end{tabular}
\end{minipage}
}
\end{center}

For permutation 3241, we have the following example. Interestingly, boxes $A$ and $C$ are the same.
\begin{center}
\noindent
\makebox[\textwidth][c]{
\begin{minipage}[t]{0.31\textwidth}
\centering
\begin{tabular}{c @{\hspace{0.5em}} r @{\,$\times$\,} c @{\,$\times$\,} l}
$D:$ & 15    & 18   & 19(12) \\
$C:$ & 13    & 16   & 18(10) \\
$B:$ & 11    & 14   & 17 \\
$A:$ & 10   & 13 & 16          
\end{tabular}
\end{minipage}
\hspace{0.06\textwidth}  
\begin{minipage}[t]{0.31\textwidth}
\centering
\begin{tabular}{c @{\hspace{0.5em}} r @{\,$\times$\,} c @{\,$\times$\,} l}
$A:$ & 13  & 16 & 19(10) \\
$D:$ & 12    & 15   & 18 \\
$B:$ & 11    & 14   & 17 \\
$C:$ & 10    & 13   & 16.          
\end{tabular}
\end{minipage}
}
\end{center}

Finally, we have already proved that 4321 is impossible. Thus, any permutation of $4$ elements except 4321 is achievable.
\end{proof}

Unsurprisingly, the only permutation of 4 boxes that is unachievable is the coolest one.

\subsection{Jointly-achievable permutations}

The physical triple Gozinta Boxes \cite{TCC} can be put in exactly two orders: black inside red inside white, and the reverse ordering. We see that in this example, only two permutations, $ABC$ and $CBA$, are jointly-achievable. Is it possible to achieve more permutations?

\begin{theorem}
\label{thm:notallsix}
	Given three boxes in any dimension, all six permutations are not jointly-achievable.
\end{theorem}

\begin{proof}
We showed in Proposition~\ref{prop:removeddimension} that if there exists a set of Gozinta Boxes in $n$ dimensions, then it exists in $n-1$ dimensions as well. We proved this by removing the largest side of each box. In the proof of the proposition, we never used the fact that the Gozinta Boxes use the reverse order as the second permutation. The same proof works for any permutation and also for any set of permutations. Thus, it is enough to prove this theorem for 2D boxes.

Suppose we have three Gozinta Rectangles $A$, $B$, and $C$, which can be arranged in any order. Without loss of generality, we assume that $C$ has the largest smallest side: $c_1 \ge \max\{a_1,b_1\}$, and, to distinguish $A$ and $B$, we also assume $a_2 \le b_2$.
 
It follows that sides $a_1$ and $b_1$ have to expand. They might expand to different values in different permutations. Thus, we introduce a new notation: $a_i'(p)$ is the value side $a_i$ expands to for the permutation $p$.

Consider a particular permutation $CBA$ when $A$ is outside $B$. Then we have $a_2 \leq b_2 < a'_1(CBA)$, so $a_2$ is the smallest side of $A$ in this permutation. Given that $a_2 \le b_2$, the side $b_2$ can not be the smallest side in this permutation. Therefore, in this permutation, the dimensions of $B$ are $(b_1'(CBA),b_2)$. As $C$ is the inner box, we can assume that $C$ is closed, and we get that $c_2 < b_2$ as we can see in the diagram below:
\begin{table}[h!]
\centering
\begin{tabular}{c @{\hspace{0.5em}} r @{\,$\times$\,}  l}
$A:$ & $a_2$ & $a_1'(CBA)$ \\
$B:$ & $b_1'(CBA)$ & \boldmath$b_2$  \\
$C:$ & $c_1$ & \boldmath$c_2$.   
\end{tabular}
\end{table}

Now we look at the permutation $BAC$. Still, side $a_2$ has to be smaller than $a'_1(BAC)$. However, the smallest side of the expanded box $C$ cannot be $c_1$ as $c_1 < a_2$, as can be seen in the case when $C$ is the innermost box. It follows that the expanded box $C$ has dimensions either $(c'_1(BAC),c_2)$ or $(c_2,c'_1(BAC))$. Both cases are shown in the diagrams below:
\begin{center}
\noindent
\makebox[\textwidth][c]{
\begin{minipage}[t]{0.31\textwidth}
\centering
\begin{tabular}{c @{\hspace{0.5em}} r @{\,$\times$\,}  l}
$C:$ & $c'_1(BAC)$ & \boldmath$c_2$ \\
$A:$ & $a_2$ & $a_1'(BAC)$  \\
$B:$ & $b_1$ & \boldmath$b_2$
\end{tabular}
\end{minipage}
\hspace{0.06\textwidth}  
\begin{minipage}[t]{0.31\textwidth}
\centering
\begin{tabular}{c @{\hspace{0.5em}} r @{\,$\times$\,}  l}
$C:$ & \boldmath$c_2$ & $c_1'(BAC)$ \\
$A:$ & \boldmath$a_2$ & $a_1'(BAC)$  \\
$B:$ & $b_1$ & $b_2$.
\end{tabular}
\end{minipage}
}
\end{center}

The left case implies that $b_2 < c_2$, which contradicts the diagram above for permutation $CBA$.

Consider the right case, which implies that $c_2>a_2$. Now, we look at the permutation $BCA$. Since $c_1$ is the expanding side of box $C$ and $b_2>c_2$, we get that $C$ has to be expanded to $(c_2,c_1'(BCA))$: 
\begin{center}
\begin{tabular}{c @{\hspace{0.5em}} r @{\,$\times$\,}  l}
$C:$ & \boldmath$c_2$ & $c_1'(BCA)$  \\
$B:$ & $b_1$ & $b_2$. 
\end{tabular}
\end{center}

From the $BAC$ diagram, we have that $a_1 \leq a_2 < c_2$, implying that box $A$ cannot be outside of the configuration above.
\end{proof}

We proved that all 6 permutations are not jointly-achievable, but we have an example where 5 permutations are jointly-achievable. The example below allows one set of boxes to be arranged in orders corresponding to all permutations except for $BAC$.

\begin{example}
\label{ex:butBAC}
Suppose boxes $A$, $B$, and $C$ have dimensions $(10, 13, 16)$, $(11, 14, 17)$, and $(9, 12, 15)$, respectively. We can achieve all permutations except $BAC$, as shown in the diagrams below:
\begin{center}
\noindent
\makebox[\textwidth][c]{
\begin{minipage}[t]{0.31\textwidth}
\centering
\begin{tabular}{c @{\hspace{0.5em}} r @{\,$\times$\,} c @{\,$\times$\,} l}
$C:$ & 12 & 15 & 18(9) \\
$B:$ & 11 & 14 & 17  \\
$A:$ & 10 & 13 & 16 \\        
\end{tabular}
\end{minipage}
\hspace{0.06\textwidth}  
\begin{minipage}[t]{0.31\textwidth}
\centering
\begin{tabular}{c @{\hspace{0.5em}} r @{\,$\times$\,} c @{\,$\times$\,} l}
$B:$ & 14 & 17 & 18(11) \\
$C:$ & 12 & 15 & 17(9) \\
$A:$ & 10 & 13 & 16 \\       
\end{tabular}
\end{minipage}
\hspace{0.06\textwidth}  
\begin{minipage}[t]{0.31\textwidth}
\centering
\begin{tabular}{c @{\hspace{0.5em}} r @{\,$\times$\,} c @{\,$\times$\,} l}
$A:$ & 13 & 16 & 19(10) \\
$C:$ & 12 & 15 & 18(9) \\
$B:$ & 11 & 14 & 17 \\     
\end{tabular}
\end{minipage}
}
\end{center}

\begin{center}
\noindent
\makebox[\textwidth][c]{
\begin{minipage}[t]{0.31\textwidth}
\centering
\begin{tabular}{c @{\hspace{0.5em}} r @{\,$\times$\,} c @{\,$\times$\,} l}
$B:$ & 11 & 14 & 17 \\
$A:$ & 10 & 13 & 16  \\
$C:$ & 9  & 12 & 15 \\        
\end{tabular}
\end{minipage}
\hspace{0.06\textwidth}  
\begin{minipage}[t]{0.31\textwidth}
\centering
\begin{tabular}{c @{\hspace{0.5em}} r @{\,$\times$\,} c @{\,$\times$\,} l}
$A:$ & 13 & 16 & 18(10) \\
$B:$ & 11 & 14 & 17 \\
$C:$ & 9  & 12 & 15.
\end{tabular}
\end{minipage}
\hspace{0.06\textwidth}  
}
\end{center}
\vspace{-2.05em}\qedhere
\end{example}

Here, we illustrate the removal of the largest side when reducing the number of dimensions by 1. We use Example~\ref{ex:butBAC} for boxes in 3D, achieving 5 permutations, and show how we can generate from them the set of 2D boxes achieving the same permutations.

\begin{example}
Suppose boxes $A$, $B$, and $C$ have dimensions $(10, 13)$, $(11, 14)$, and $(9, 12)$, respectively. We can achieve all permutations except $BAC$, as shown in the diagrams below:
\begin{center}
\noindent
\makebox[\textwidth][c]{
\begin{minipage}[t]{0.31\textwidth}
\centering
\begin{tabular}{c @{\hspace{0.5em}} r @{\,$\times$\,} l}
$C:$ & 12 & 15(9) \\
$B:$ & 11 & 14  \\
$A:$ & 10 & 13 \\        
\end{tabular}
\end{minipage}
\hspace{0.06\textwidth}  
\begin{minipage}[t]{0.31\textwidth}
\centering
\begin{tabular}{c @{\hspace{0.5em}} r @{\,$\times$\,} l}
$B:$ & 14 & 17(11) \\
$C:$ & 12 & 15(9) \\
$A:$ & 10 & 13 \\       
\end{tabular}
\end{minipage}
\hspace{0.06\textwidth}  
\begin{minipage}[t]{0.31\textwidth}
\centering
\begin{tabular}{c @{\hspace{0.5em}} r @{\,$\times$\,} l}
$A:$ & 13 & 16(10) \\
$C:$ & 12 & 15(9) \\
$B:$ & 11 & 14 \\     
\end{tabular}
\end{minipage}
}
\end{center}

\begin{center}
\noindent
\makebox[\textwidth][c]{
\begin{minipage}[t]{0.31\textwidth}
\centering
\begin{tabular}{c @{\hspace{0.5em}} r @{\,$\times$\,} l}
$B:$ & 11 & 14 \\
$A:$ & 10 & 13  \\
$C:$ & 9  & 12 \\        
\end{tabular}
\end{minipage}
\hspace{0.06\textwidth}  
\begin{minipage}[t]{0.31\textwidth}
\centering
\begin{tabular}{c @{\hspace{0.5em}} r @{\,$\times$\,} l}
$A:$ & 13 & 16(10) \\
$B:$ & 11 & 14 \\
$C:$ & 9  & 12.
\end{tabular}
\end{minipage}
\hspace{0.06\textwidth}  
}
\end{center}
\vspace{-2.05em}\qedhere
\end{example}

Surprisingly, the coolest permutation, $CBA$, is included in our examples, while a less cool permutation is not. However, the following proposition shows that each permutation can be the one that is excluded.

\begin{proposition}
 Any set of five permutations of three boxes can be jointly-achievable.
\end{proposition}

\begin{proof}
In Example~\ref{ex:butBAC}, we achieve 5 permutations with the same set of boxes. The excluded permutation is $BAC$. If we relabel our boxes, we can make $BAC$ into any other permutation. 
\end{proof}

We are interested in the special case when two boxes are the same. Then, these two boxes can always be swapped, implying that the number of possible permutations is even. It follows from Theorem~\ref{thm:notallsix} that the number of joint permutations in this case is 2 or 4. The example below shows a set of boxes that allows 4 permutations.

\begin{example}
Consider our main example of triple Gozinta Boxes from Example~\ref{ex:6-8-10}, where two boxes $A$ and $C$ have dimensions $(6, 8, 10)$, and box $B$ has dimensions $(7, 9, 11)$. The diagrams below show permutations $ABC$ and $ACB$:
\begin{center}
\noindent
\makebox[\textwidth][c]{
\begin{minipage}[t]{0.31\textwidth}
\centering
\begin{tabular}{c @{\hspace{0.5em}} r @{\,$\times$\,} c @{\,$\times$\,} l}
$C:$ & 8 & 10 & 12(6) \\
$B:$ & 7 & 9 & 11 \\
$A:$ & 6  & 8 & 10
\end{tabular}
\end{minipage}
\hspace{0.06\textwidth}  
\begin{minipage}[t]{0.31\textwidth}
\centering
\begin{tabular}{c @{\hspace{0.5em}} r @{\,$\times$\,} c @{\,$\times$\,} l}
$B:$ & 9 & 11 & 14(7) \\
$C:$ & 8 & 10 & 12(6) \\
$A:$ & 6  & 8 & 10.
\end{tabular}
\end{minipage}
\hspace{0.06\textwidth}  
}
\end{center}
By symmetry, permutations $CBA$ and $CAB$ are also achievable, making 4 jointly-achievable permutations.
\end{example}

\section{From one set of boxes to another set of boxes}
\label{sec:FromTo}

Suppose we have a set of Gozinta Boxes in any dimension that can be put into two or more different orders. We are interested in finding operations on the side lengths of the boxes in the set so that we can build a new set of boxes that can be used to perform the same trick.

\begin{proposition}
\label{prop:FromTo}
Consider a set of Gozinta Boxes that can be put into several orders. If we perform one of the following operations on the sides of the boxes, we will get a set of Gozinta Boxes that can be put in the same set of orders as the original set of boxes.
\begin{enumerate}[label=Operation \arabic*:, leftmargin=*, labelindent=1em]
        \item Multiply all sides of all boxes by a constant.
        \item Add a constant positive number to all sides of all boxes.
    \end{enumerate}
\end{proposition}

\begin{proof}
A set of Gozinta Boxes can be put into a set of particular orders if and only if a specific set of inequalities is satisfied. These inequalities are of the form $s_1 < s_2$, $s'_1 < s_2$, $s_1 < s'_2$, and $s'_1 < s'_2$. The dimensions also need to satisfy inequalities of the form $s'_i < 2s_i$.

When we multiply each closed and expanded side by the same constant, all inequalities will still hold. Thus, the new boxes can be put in the same orders as the original boxes.

We can also add a constant $c > 0$ to each closed and expanded side. Moreover, if a side $a_i$ can expand to $a'_i$, then the side $a_i +c$ can expand to $a'_i + c$. Indeed, if $a'_i < 2a_i$, then $a'_i + c < 2a_i + c < 2(a_i+c)$. This means that adding a constant does not affect the expansion bound. It also does not affect any of the other inequalities. Therefore, we can put the boxes in the same orders as the original boxes.
\end{proof}

Suppose we have a set of Gozinta Boxes that could have been put in a particular order if the expansion bound was not in place. Namely, the Gozinta Boxes could have been put in some order, if we allow the expansion $a'_i$ such that $a'_i \geq 2a_i$. Then, we can find a set that satisfies the bound constraint and can be put in the given order.

\begin{proposition}
\label{prop:ignoreexpansionconstraint}
While looking for Gozinta Boxes that can be put in particular orders, we can ignore the expansion bound $s_i' \leq 2s_i$.
\end{proposition}

\begin{proof}
Suppose there exists a set of Gozinta Boxes that can be put in particular orders, except the expansion bound $s'_i < 2s_i$ might not be satisfied for some sides of some boxes. Suppose $m$ is greater than the maximum of all such $s_i$ and $s'_i$. Then, by Proposition~\ref{prop:FromTo}, we can add $m$ to all the sides of all the boxes, creating a set of Gozinta Boxes that can be put in the same orders with the expansion bound satisfied. Indeed, for each side we get $s'_i + m < 2m < 2(s_i+m)$.
\end{proof}

We now add a third operation that we can use. We use Proposition~\ref{prop:ignoreexpansionconstraint} to ignore the expansion bound.

\begin{proposition}
Consider a set of Gozinta Boxes that can be put into several particular orders. In addition, we assume that any dimension can be extended to the largest possible dimension needed. If we perform the following operation on the sides of the boxes, we will get a set of Gozinta Boxes that can be put in the same set of orders as the original set of boxes.
\begin{enumerate}[label=Operation \arabic*:, leftmargin=*, labelindent=1em]
  \setcounter{enumi}{2} 
        \item If there exists a side with length $s$ and two numbers $x$ and $y$ such that $x \leq s \leq y$ and any other side closed or expanded is either less than $x$ or more than $y$, then we can replace $s$ with any number between $x$ and $y$, inclusive.
    \end{enumerate}
\end{proposition}

\begin{proof}
Consider box $A$ with side $s$. Suppose we replace $s$ with a number $r$ between $x$ and $y$, making a new box $B$. Then, if some box fits inside of $A$, it also fits inside $B$ because all the sides that need to be less than $s$ are less than $r$. Similarly, if box $A$ fits inside another box, box $B$ also fits inside because all the sides that need to be greater than $s$ are greater than $r$ too.
\end{proof}

The analysis in this section implies that for large enough boxes, we only care about the relative lengths of their sides. In the next section, we classify two Gozinta Boxes up to the relative ordering of their sides.

\section{Two boxes}
\label{sec:twoboxes}

We consider two Gozinta boxes, $A$ and $B$, that can be put in natural and reverse orders. Without loss of generality, we assume that $a_1 \leq b_1$. 

In Section~\ref{sec:FromTo}, we discussed how we could modify a set of Gozinta Boxes to make a different set of Gozinta Boxes, as well as how to get rid of the expansion bound.

\begin{example}
In Example~\ref{ex:largestsideexpands}, we had two boxes with dimensions $(5, 7, 999)$ and $(6, 8, 500)$, where the second box has to expand on the largest side. We can build a new set of boxes by adding 1000 to each side. For the new pair of boxes $(1005, 1007, 1999)$ and $(1006, 1008, 1500)$ any side in the second box can expand.
\end{example}

Now, we discuss all the possibilities for the relative orders of the sides of boxes $A$ and $B$. First, we simplify our potential cases.

\begin{lemma}
\label{lem:2boxes}
Looking for possible relative orders of sides, we can assume that the smallest side expands to the largest side for both boxes.
\end{lemma}

\begin{proof}
When box $B$ is on the outside, the dimensions of the expanded $B$ in non-decreasing order could be one of $(b'_1,b_2,b_3)$, $(b_2,b'_1,b_3)$, $(b_2,b_3,b'_1)$, $(b_1,b'_2,b_3)$, $(b_1,b_3,b'_2)$, and $(b_1,b_2,b'_3)$. We notice that $(b_2,b_3,b'_1) \succ (b_2,b'_1,b_3) \succ (b'_1,b_2,b_3)$ and $(b_1,b_3,b'_2) \succ (b_1,b'_2,b_3)$. Thus, we can assume that the expanded side becomes the largest, and the dimensions of the expanded $B$ are one of $(b_2,b_3,b'_1)$, $(b_1,b_3,b'_2)$, and $(b_1,b_2,b'_3)$. By Proposition~\ref{prop:ignoreexpansionconstraint}, we can ignore the expansion bounds. This means we can assume that the smallest side can expand to anything we need. In particular, it can expand to the maximum of $b'_1$, $b'_2$, and $b'_3$ above. In this case, we have $(b_2,b_3,b'_1) \succeq (b_1,b_3,b'_2) \succeq (b_1,b_2,b'_3)$. This means if box $A$ fits inside box $B$ with the largest side of box $B$ expanded, it also fits inside box $B$ with the smallest side expanded.
\end{proof}

Now we are ready to classify two 3D Gozinta Boxes based on the relative ordering of their sides.

\begin{theorem}
	If two 3D Gozinta Boxes $A$ and $B$ can fit inside each other, and $a_1 \leq b_1$, then the dimensions can be divided into the following four types, where in each inequality chain, no two adjacent inequalities can both be equalities:
	\begin{enumerate}[label=Type \arabic*:, leftmargin=*, labelindent=1em]
	\item $a_1 \le b_1 < a_2  \leq b_2 < a_3 \le b_3$,
	\item $a_1 \le b_1 < a_2 \le b_2 \le b_3<a_3$,
	\item $a_1 \le b_1 \leq b_2 < a_2 \leq a_3 \le b_3$,
	\item $a_1 \leq b_1 \leq b_2 < a_2 < b_3 < a_3$.
	\end{enumerate}
\end{theorem}

\begin{proof}
By Lemma~\ref{lem:2boxes}, we can assume that the following conditions are true:
\[(a_2,a_3,a'_1) \succ (b_1,b_2,b_3) \quad \textrm{ and } \quad (b_2,b_3,b_1') \succ (a_1,a_2,a_3).\]
Combining these inequalities, we get $b_1 < a_2 < b_3$ and $a_1 < b_2 < a_3$.

Now, we consider cases depending on whether $a_2 \leq b_2$ and $a_3 \leq b_3$.

Suppose $a_2 \leq b_2$ and $a_3 \leq b_3$. Thus, we get $a_1 \le b_1 < a_2  \leq b_2 < a_3 \le b_3$, which is Type 1.

Suppose $a_2 \leq b_2$ and $a_3 < b_3$. Thus, we get $a_1 \le b_1 < a_2 \le b_2 \le b_3 < a_3$. In addition, $a_2 < b_3$, so we cannot have two equalities in a row.

Suppose $a_2 < b_2$ and $a_3 \geq b_3$. Thus, we get $a_1 \le b_1 \leq b_2 < a_2 \leq a_3 \le b_3$. In addition, $a_1 < b_2$ and $a_2 < b_3$, so we cannot have two equalities in a row.

Suppose $a_2 < b_2$ and $a_3 < b_3$. Thus, we get $a_1 \leq b_1 \leq b_2 < a_2 < b_3 < a_3$. In addition, $a_1 < b_2$, so we cannot have two equalities in a row.
\end{proof}

We also have examples for each type.

\begin{example}
In Example~\ref{ex:6-8-10}, box $A$ has dimensions $(6,8,10)$, and box $B$ has dimensions $(7,9,11)$. These boxes satisfy $a_1<b_1<a_2<b_2<a_3<b_3$ and are of Type 1.

In Example~\ref{ex:6-9-10}, box $A$ has dimensions $(6, 9, 10)$, and box $B$ has dimensions $(7, 8, 11)$. These boxes satisfy $a_1 <  b_1 < b_2 < a_2 < a_3 < b_3$ and are of Type 3.

In Example~\ref{ex:5-11-13versus7-10-12}, box $A$ has dimensions $(5, 11, 13)$, and box $B$ has dimensions $(7, 10, 12)$. These boxes satisfy $a_1 <  b_1 < b_2 < a_2 < b_3 < a_3$ and are of Type 4.
\end{example}

We also show examples where some sides are equal.

\begin{example}
	 In Example~\ref{ex:3-4-5}, both boxes have dimensions $(3,4,5)$ and they satisfy $a_1 = b_1 < a_2 = b_2 < a_3 = b_3$, which is Type 1 in our classification. It is borderline to any other type.
\end{example}

Here is an example for $b_1=b_2$ and $a_2=a_3$ combined in the same set of boxes. 

\begin{example}
Box $A$ has dimensions $(4, 6, 6)$, and box $B$ has dimensions $(5, 5, 7)$. These boxes satisfy $a_1 < b_1 = b_2 < a_2 = a_3 < b_3$ and are of Type 3. The diagram below shows the natural order on the left and the reverse order on the right:
\begin{center}
\noindent
\makebox[\textwidth][c]{
\begin{minipage}[t]{0.31\textwidth}
\centering
\begin{tabular}{c @{\hspace{0.5em}} r @{\,$\times$\,} c @{\,$\times$\,} l}
$B:$ & 5 & 7 & 7(5) \\
$A:$ & 4 & 6 & 6     
\end{tabular}
\end{minipage}
\hspace{0.06\textwidth}  
\begin{minipage}[t]{0.31\textwidth}
\centering
\begin{tabular}{c @{\hspace{0.5em}} r @{\,$\times$\,} c @{\,$\times$\,} l}
$A:$ & 6 & 6 & 8(4) \\
$B:$ & 5 & 5 & 7.
\end{tabular}
\end{minipage}
}
\end{center}
\vspace{-2.05em}\qedhere
\end{example}

\section{Acknowledgments}

We are grateful to the PRIMES STEP program for giving us the opportunity to conduct this research.

\end{document}